\documentclass[12pt]{article}

\usepackage{amssymb}
\usepackage{amsthm}
\usepackage{latexsym}
\usepackage{tikz}
\usepackage{amsfonts}
\usepackage{tkz-graph}
\usepackage{amssymb}
\usepackage{comment}
\usepackage{tikz}

\usetikzlibrary{shapes}
\usetikzlibrary{snakes}
\usetikzlibrary{shapes}
\usetikzlibrary{arrows}
\usetikzlibrary{decorations.pathmorphing}

\newtheorem{Theorem}{Theorem}[section]

\newtheorem{Definition}[Theorem]{Definition}
\newtheorem{Lemma}[Theorem]{Lemma}
\newtheorem{Observation}[Theorem]{Observation}
\newtheorem{OpenProblem}[Theorem]{Open Problem}

\def\inst#1{$^{#1}$}

\title{Chromatic symmetric functions and $H$-free graphs}

\author{
Ang\`ele M. Hamel\inst{1}
    \and Ch\'inh T. Ho\`ang\inst{1}
    \and Jake E. Tuero \inst{1}  }
\begin{document}
\maketitle

\begin{center}
{\footnotesize

\inst{1} Department of Physics and Computer Science, Wilfrid Laurier
University, \\Waterloo, Ontario, Canada }

\end{center}

\begin{abstract}
Two celebrated conjectures in chromatic symmetric function theory concern the positivity chromatics symmetric functions of claw-free graphs.  Here we extend the claw-free idea to general graphs and consider the e-positivity question for $H$-free graphs where $H = \{claw, F\}$ and $H=$\{{\em claw}, $F$, {\em co-}$F$\},  where  $F$ is a four-vertex graph.  We settle the question for all cases except $H=$\{{\em claw, co-diamond}\}, and we provide some partial results in that case.
\end{abstract}

\section{Introduction}

A key area of investigation in symmetric function theory concerns the  e-positivity, and/or Schur positivity, of a particular class of symmetric functions. In chromatic symmetric function theory there are two celebrated conjectures in this regard that focus on claw-free graphs, where the claw is the four-vertex  bipartite graph $K_{1,3}$ (see Figure \ref{fig:4vertex-graphs}). One conjecture, due to Stanley and Stembridge \cite{SS}, hypothesizes   that the chromatic symmetric function of a claw-free incomparabilty graph is e-positive. 
The second conjecture, due to Stanley \cite{S} with credit to Gasharov, hypothesizes that the chromatic symmetric function of a claw-free graph is Schur positive (for definitions see Section \ref{Background}).


In parallel to this, in graph theory, much effort has been spent in characterizing the chromatic characteristics of graphs that are $H$-free, where $H$ is some set of induced subgraphs.  A key question in this domain is, can  the chromatic number of a $H$-free graph be determined in polynomial time?  This answer is known for large classes of graphs, but, interestingly, the classification for all combinations of subgraphs with four vertices (of which the claw is obviously one) is not yet complete \cite{LM}. 
Thus, combining these two areas, it is a natural question to generalize the  claw-free conjectures and ask about other $H$-free graphs: ``For which $H$-free graphs are their chromatic symmetric functions e-positive?" and to ask in particular, ``For which $H$-free graphs are their chromatic symmetric functions  e-positive, where $H$ is the claw along with one other four-vertex graphs?"  Obviously the corresponding Schur positivity questions would be subcases of the Gasharov-Stanley conjecture.

Formally, let $H$ be a set of four-vertex graphs. By abuse of notation we say a graph is e-positive if its chromatic symmetric function is e-positive. Then we want to know for which sets $H$ of four-vertex graphs, an $H$-free graph is e-positive (or not). There are eleven graphs on four vertices, see Figure \ref{fig:4vertex-graphs}. The claw is not e-positive, since $X_{K_{1,3}} = e_4+5e_{3,1} -2e_{2,2} + e_{2,1,1}$, and we include it always as part of $H$ (as an aside, note that some intuition as to why the claw is not $e$-positive can be gained by considering partition dominance and the concept of {\em nice} graphs, as defined in Stanley \cite{Sold}). We consider two different directions.

First, $H=$\{claw, $F$\} for $F$ a single four-vertex graph. By using an example of Stanley, given here in Figure \ref{fig:ThreeSun}, we see that this example does not contain diamond, co-claw, $K_4$, $4K_1$, $C_4$, or $2K_2$, and we can thus show that a graph that is free of claw paired with each of these is not necessarily e-positive.
The remaining four-vertex graphs are $P_4$, paw, co-paw, and co-diamond.
By a result of Tsujie \cite{Tsujie}, if $H$ contains a claw and a $P_4$, then $H$-free graphs are e-positive.
In this paper we show that if $H$ contains a paw, or if $H$ contains a co-paw, then $H$-free graphs are e-positive.  Thus the only outstanding case of $H=$\{claw, $F$\} is $H=$\{claw, co-diamond\}.
In Section \ref{sec:clawF} we present some partial results on this case.
We summarize the $H=$\{{\em claw, $F$}\} results in Table \ref{tablecases}.

\begin{figure}\label{tablecases}
\begin{center}
\begin{tabular}{||c|c|c||}
\hline \hline
Set  $H$ & Positivity & Reference\\
\hline \hline
claw, $P_4$ & e-positive &  \cite{Tsujie} \\  \hline
claw, paw & e-positive &  Theorem \ref{thm:claw_paw_free} \\  \hline
claw, co-paw & e-positive &  Theorem \ref{thm:claw_copaw_free} \\  \hline
claw, co-diamond & unknown   & unknown \\ \hline
claw, diamond  & not  necessarily e-positive & Lemma \ref{thm:somethm} \\  \hline
claw, co-claw & not necessarily e-positive & Lemma \ref{thm:somethm} \\  \hline
claw, $K_4$ & not necessarily e-positive & Lemma \ref{thm:somethm} \\  \hline
claw, $4K_1$ & not necessarily e-positive & Lemma \ref{thm:somethm} \\  \hline
claw, $C_4$ & not necessarily e-positive & Lemma \ref{thm:somethm} \\  \hline
claw, $2K_2$ & not necessarily e-positive & Lemma \ref{thm:somethm} \\  
\hline \hline
\end{tabular}
\end{center}
\caption{Table of e-positivity of pairs of four-vertex graphs}
\end{figure}

Second, $H=$\{{\em claw}, $F$, co-$F$\}, where $F$ is a single four-vertex graph and co-$F$ is the graph complement of $F$.   Again, the Stanley example, Figure \ref{fig:ThreeSun}, takes care of a number of cases and our theorems in Section \ref{sec:clawF} handle all but one of the rest. The last case, $H=$\{{\em claw, diamond, co-diamond}\} is handled by a case-by-case argument.

\begin{figure}\label{tablecases2}
\begin{center}
\begin{tabular}{||c|c|c||}
\hline \hline
Set  $H$ & Positivity & Reference\\
\hline \hline
claw, $P_4$ & e-positive &  \cite{Tsujie} \\  \hline
claw, paw, co-paw & e-positive &  Theorem \ref{thm:clawpawcopaw} \\  \hline
claw, diamond, co-diamond & e-positive &  Theorem \ref{thm:clawdiamondcodiamond} \\  \hline
claw, co-claw & not necessarily e-positive & Lemma \ref{thm:somethm} \\  \hline
claw, $C_4, 2K_2$ & not necessarily e-positive &  Theorem \ref{thm:negative} \\  \hline
claw, $K_4, 4K_1$ & not necessarily e-positive &  Theorem \ref{thm:negative} \\  \hline
\hline \hline
\end{tabular}
\end{center}
\caption{Table of e-positivity of claw plus $F$ and co-$F$.}
\end{figure}




\begin{figure}
\begin{center}
\begin{tikzpicture} [scale = 1.25]
\tikzstyle{every node}=[font=\small]

\newcommand{\size}{1}

\newcommand{\p4}{1}{
    \path (\size * 4, 0) coordinate (g1);
    \path (g1) +(-\size, 0) node (g1_1){};
    \path (g1) +(0, 0) node (g1_2){};
    \path (g1) +(\size, 0) node (g1_3){};
    \path (g1) +(2 * \size, 0) node (g1_4){};
    \foreach \Point in {(g1_1), (g1_2), (g1_3), (g1_4)}{
        \node at \Point{\textbullet};
    }
    \draw   (g1_1) -- (g1_2)
            (g1_2) -- (g1_3)
            (g1_3) -- (g1_4);
    \path (g1) ++(\size  / 2,-\size / 2) node[draw=none,fill=none] { {\large $P_4$}};
}

\newcommand{\kfour}{2}{
    \path (0, - \size * 2) coordinate (g2);
    \path (g2) +(0, 0) node (g2_1){};
    \path (g2) +(0, \size) node (g2_2){};
    \path (g2) +(\size, \size) node (g2_3){};
    \path (g2) +(\size, 0) node (g2_4){};
    \foreach \Point in {(g2_1), (g2_2), (g2_3), (g2_4)}{
        \node at \Point{\textbullet};
    }
    \draw   (g2_1) -- (g2_2)
            (g2_1) -- (g2_3)
            (g2_1) -- (g2_4)
            (g2_2) -- (g2_3)
            (g2_2) -- (g2_4)
            (g2_3) -- (g2_4);
    \path (g2) ++(\size  / 2,-\size / 2) node[draw=none,fill=none] { {\large $K_4$}};
}

\newcommand{\diam}{3}{
    \path (\size * 2, - \size * 2) coordinate (g3);
    \path (g3) +(0, 0) node (g3_1){};
    \path (g3) +(0, \size) node (g3_2){};
    \path (g3) +(\size, \size) node (g3_3){};
    \path (g3) +(\size, 0) node (g3_4){};
    \foreach \Point in {(g3_1), (g3_2), (g3_3), (g3_4)}{
        \node at \Point{\textbullet};
    }
    \draw   (g3_1) -- (g3_2)
            (g3_1) -- (g3_3)
            (g3_1) -- (g3_4)
            (g3_2) -- (g3_3)
            (g3_3) -- (g3_4);
    \path (g3) ++(\size  / 2,-\size / 2) node[draw=none,fill=none] { {\large $diamond$}};
}

\newcommand{\cfour}{4}{
    \path (\size * 4, - \size * 2) coordinate (g4);
    \path (g4) +(0, 0) node (g4_1){};
    \path (g4) +(0, \size) node (g4_2){};
    \path (g4) +(\size, \size) node (g4_3){};
    \path (g4) +(\size, 0) node (g4_4){};
    \foreach \Point in {(g4_1), (g4_2), (g4_3), (g4_4)}{
        \node at \Point{\textbullet};
    }
    \draw   (g4_1) -- (g4_2)
            (g4_1) -- (g4_4)
            (g4_2) -- (g4_3)
            (g4_3) -- (g4_4);
    \path (g4) ++(\size  / 2,-\size / 2) node[draw=none,fill=none] { {\large $C_4$}};
}

\newcommand{\paw}{5}{
    \path (\size * 6, - \size * 2) coordinate (g5);
    \path (g5) +(0, 0) node (g5_1){};
    \path (g5) +(0, \size) node (g5_2){};
    \path (g5) +(\size, \size) node (g5_3){};
    \path (g5) +(\size, 0) node (g5_4){};
    \foreach \Point in {(g5_1), (g5_2), (g5_3), (g5_4)}{
        \node at \Point{\textbullet};
    }
    \draw   (g5_1) -- (g5_2)
            (g5_2) -- (g5_3)
            (g5_2) -- (g5_4)
            (g5_3) -- (g5_4);
    \path (g5) ++(\size  / 2,-\size / 2) node[draw=none,fill=none] { {\large $paw$}};
}

\newcommand{\claw}{6}{
    \path (\size * 8, - \size * 2) coordinate (g6);
    \path (g6) +(0, 0) node (g6_1){};
    \path (g6) +(0, \size) node (g6_2){};
    \path (g6) +(\size, \size) node (g6_3){};
    \path (g6) +(\size, 0) node (g6_4){};
    \foreach \Point in {(g6_1), (g6_2), (g6_3), (g6_4)}{
        \node at \Point{\textbullet};
    }
    \draw   (g6_1) -- (g6_2)
            (g6_2) -- (g6_4)
            (g6_2) -- (g6_3);
    \path (g6) ++(\size  / 2,-\size / 2) node[draw=none,fill=none] { {\large $claw$}};
}

\newcommand{\cokfour}{7}{
    \path (0, - \size * 4) coordinate (g7);
    \path (g7) +(0, 0) node (g7_1){};
    \path (g7) +(0, \size) node (g7_2){};
    \path (g7) +(\size, \size) node (g7_3){};
    \path (g7) +(\size, 0) node (g7_4){};
    \foreach \Point in {(g7_1), (g7_2), (g7_3), (g7_4)}{
        \node at \Point{\textbullet};
    }

    \path (g7) ++(\size  / 2,-\size / 2) node[draw=none,fill=none] { {\large $4K_1$}};
}

\newcommand{\codiamond}{8}{
    \path (\size * 2, - \size * 4) coordinate (g8);
    \path (g8) +(0, 0) node (g8_1){};
    \path (g8) +(0, \size) node (g8_2){};
    \path (g8) +(\size, \size) node (g8_3){};
    \path (g8) +(\size, 0) node (g8_4){};
    \foreach \Point in {(g8_1), (g8_2), (g8_3), (g8_4)}{
        \node at \Point{\textbullet};
    }
    \draw   (g8_3) -- (g8_4);
    \path (g8) ++(\size  / 2,-\size / 2) node[draw=none,fill=none] { {\large co-$diamond$}};
}

\newcommand{\cocfour}{9}{
    \path (\size * 4, - \size * 4) coordinate (g8);
    \path (g8) +(0, 0) node (g8_1){};
    \path (g8) +(0, \size) node (g8_2){};
    \path (g8) +(\size, \size) node (g8_3){};
    \path (g8) +(\size, 0) node (g8_4){};
    \foreach \Point in {(g8_1), (g8_2), (g8_3), (g8_4)}{
        \node at \Point{\textbullet};
    }
    \draw (g8_1) -- (g8_2)
          (g8_3) -- (g8_4);
    \path (g8) ++(\size  / 2,-\size / 2) node[draw=none,fill=none] { {\large $2K_2$}};
}

\newcommand{\copaw}{10}{
    \path (\size * 6, - \size * 4) coordinate (g9);
    \path (g9) +(0, 0) node (g9_1){};
    \path (g9) +(0, \size) node (g9_2){};
    \path (g9) +(\size, \size) node (g9_3){};
    \path (g9) +(\size, 0) node (g9_4){};
    \foreach \Point in {(g9_1), (g9_2), (g9_3), (g9_4)}{
        \node at \Point{\textbullet};
    }
    \draw   (g9_1) -- (g9_2)
            (g9_2) -- (g9_3);
    \path (g9) ++(\size  / 2,-\size / 2) node[draw=none,fill=none] { {\large co-$paw$}};
}

\newcommand{\coclaw}{11}{
    \path (\size * 8, - \size * 4) coordinate (g10);
    \path (g10) +(0, 0) node (g10_1){};
    \path (g10) +(0, \size) node (g10_2){};
    \path (g10) +(\size, \size) node (g10_3){};
    \path (g10) +(\size, 0) node (g10_4){};
    \foreach \Point in {(g10_1), (g10_2), (g10_3), (g10_4)}{
        \node at \Point{\textbullet};
    }
    \draw   (g10_1) -- (g10_2)
            (g10_1) -- (g10_3)
            (g10_2) -- (g10_3);
    \path (g10) ++(\size  / 2,-\size / 2) node[draw=none,fill=none] { {\large co-$claw$}};
}

\end{tikzpicture}
\end{center}
\caption{All four-vertex graphs} \label{fig:4vertex-graphs}
\end{figure}
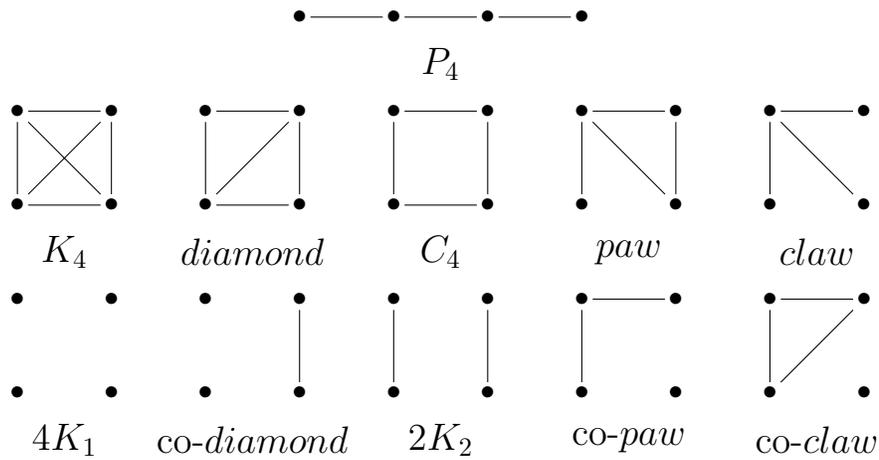

\section{Background and Notation}
\label{Background}

 The original interest in claw-free graphs in the context of chromatic symmetric functions stems from Conjecture 5.1 of Stanley and Stembridge \cite{SS}.  This interest can be traced  farther back to another historically significant conjecture, the Goulden-Jackson immanant conjecture \cite{GJ},  that hypothesizes that the expansion of the immanant of the Jacobi-Trudi matrix has nonnegative coefficients
(the immanant is similar to a determinant but with the sign replaced by another symmetric group character).
This conjecture was proved soon after by Curtis Greene \cite{Gr}.
Related to this, Stanley and Stembridge \cite{SS} investigated conjectures involving the positivity of coefficients in immanants
and conjectured  that  if  a poset is $(3+1)$-free, then the chromatic symmetric function of  its incomparability graph is e-positive (see Stanley \cite{Sold} or Stanley and Stembridge \cite{SS} for definitions).


 Gasharov \cite{G} proved that the chromatic symmetric functions of the incomparability graphs of  $(3+1)$-free posets are Schur positive. Note that if a symmetric function is e-positive it is necessarily Schur positive, but not the other way around, so this is a weaker result than the one asked for.
 The e-positivity of the chromatic symmetric function of the incomparability graph of a $(3+1)$-free poset  remains open, although there has been much progress on this question, see Guay-Paquet \cite{GP}. 

We now turn to some definitions. Further details on definitions and notation can be found in Stanley \cite{StanleyBook}.
A coloring of the set of vertices $V$ of some graph $G=(V, E)$ is a function $\kappa$ from $V$ to the positive integers $\mathbb{Z}^+$: $\kappa:V\rightarrow \mathbb{Z}^+$. A coloring $\kappa$ is proper if $\kappa(u) \not= \kappa(v)$ whenever vertex $u$ is not adjacent to vertex $v$. 
Chromatic symmetric functions were defined by Stanley \cite{Sold} as a generalization of the chromatic polynomial.  Indeed, if $x_1=x_2=x_3=\ldots =1$, this expression reduces to the chromatic polynomial for a graph.

\begin{Definition}
For a graph $G$ with vertex set $V=\{v_1, v_2, \ldots , v_N\}$ and edge set $E$, the chromatic symmetric function is defined to be
\[
X_G=\sum_{\kappa} x_{\kappa(v_1)} x_{\kappa(v_2)}\cdots x_{\kappa(v_{N})}
\]
 where the sum is over all proper colorings $\kappa$ of $G$. 
 \end{Definition}

A partition $\lambda=(\lambda_1, \lambda_2, \ldots \lambda_{\ell})$ of a positive integer $n$ is a nondecreasing sequence of positive integers: $\lambda_1 \geq \lambda_2 \geq \ldots \geq \lambda_{\ell}$, where  $\lambda_i$ is called the $i$th part of $\lambda$, $1\leq i \leq \ell$.   The transpose, $\lambda'$, of $\lambda$, is defined by its parts: $\lambda_i' =  | \{j: \lambda_j \geq i\} |$. The elementary symmetric function, $e_i(x)$, is defined as 
\[
e_i(x) = \sum_{j_1 < j_2< \cdots < j_i} x_{j_{1}} \cdots x_{j_{i}}
\]
and the elementary symmetric function, $e_\lambda(x)$, is defined as $e_\lambda(x) = e_{\lambda_{1}} e_{\lambda_{2}} \ldots e_{\lambda_{\ell}}$.

 The Schur function is another common symmetric functions.  There are a number of useful definitions of the Schur functions, but the easiest in this context is one in terms of the elementary symmetric functions:  $s_\lambda (x) = det(e_{\lambda_{i}'-i+j} (x) )_{1 \leq i, j \leq \lambda_1}$  where if $\lambda_{i}'-i+j < 0$ then $e_{\lambda_{i}'-i+j} (x)=0$.
 
 Both the set of elementary symmetric functions and the set of Schur functions each form a basis for the algebra of symmetric functions (for details, see Stanley\cite{StanleyBook}).  If a given symmetric function can be written as a nonnegative linear combination of elementary (resp.\ Schur) symmetric functions we say the symmetric function is {\em e-positive} (resp.\ {\em Schur positive}).
 
In graph theory we often discuss the idea of an $H$-free graph. Define $H$-free to be the class of graphs that do not contain any graph is $H$ as an induced subgraph.  Let $P_k$ be the chordless path on $k$ vertices and $C_k$ be the chordless cycle on $k$ vertices.  The complete graph $K_n$ is the graph on $n$ vertices such that there is an edge between all pairs of vertices.  The graph $K_3$ is called the {\em triangle}, and it's complement, which is $3K_1$, three vertices with no edges at all, is called the {\em co-triangle}.  A {\em stable set} is a set $S$ of vertices of a graph such that there are no edges between any of the vertices in $S$, e.g. a co-triangle is a stable set of size $3$. The notation $\alpha(G)$ denotes the number of vertices in the largest stable set of $G$.

We will make use of a number of known results concerning chromatic symmetric functions and e-positivity.  These lemmas are as follows:


	\begin{Lemma} \cite[Proposition 2.3]{Sold} \label{lem:chromatic_disjoint_union}
		If a finite simple graph $G$ is a disjoint union of subgraphs $G_1 \cup  G_2$, then $X_G = X_{G_{1}}X_{G_{2}}$. \\
	\end{Lemma}

	\begin{Lemma}\cite[Theorem 8]{CV}\label{lem:chromatic_complete}
		$X_{K_n} = n! e_n$. \\
	\end{Lemma}

	\begin{Lemma}\cite[Proposition 5.3]{Sold}
	\label{lem:e_path}
		$X_{P_k}$ is $e$-positive. \\
	\end{Lemma}

	\begin{Lemma} \cite[Proposition 5.4]{Sold}
	\label{lem:e_cycle}
		$X_{C_k}$ is $e$-positive. \\
	\end{Lemma}
	
\begin{Lemma} \cite[Exercise 7.47j]{StanleyBook}
If $G$ is co-triangle-free, $G$ is e-positive.
\label{lem:cotrianglefreeepos}
\end{Lemma}

Moreover, we also employ a number of results from the graph theory literature:

	\begin{Theorem} \cite[Theorem 1.5]{Tsujie} \label{thm:P4_free}
	If $G$ is (claw, $P_4$)-free, then $X_G$ is e-positive.
\end{Theorem}

	\begin{Theorem} \cite{O}\label{thm:olariu}
			If $G$ is paw-free, then every connected component of $G$ is triangle-free or is a complete multipartite graph. \\
	\end{Theorem}

We will also need the following statement of the structure of (claw,triangle)-free graphs.
	\begin{Lemma}\label{lem:claw_triangle_free}
		If $G$ is (claw, triangle)-free, then each component of $G$ is a chordless path or cycle, i.e. $G=P_k$ or $G=C_k$. 
	\end{Lemma}
{\it Proof of Lemma~\ref{lem:claw_triangle_free}}. Let $T$ be a component of $G$. Suppose $T$ contains a chordless cycle $C$. Since $G$ is triangle-free, we may assume $C$ has at least four vertices. If every vertex of $T$ lies on $C$, then we are done. So we may assume there is a vertex $v \in T -C$ that has a neighbor $c$ in $C$. If $v$ is adjacent to a neighbor $c'$ of $c$ on $C$, then $T$ contains a triangle with vertices $v, c, c'$; otherwise $c$ is the center of a claw in $T$. Now, we may assume $T$ contains no cycle, that is, $T$ is a tree, but a claw-free tree has maximum degree at most two and so is necessarily a chordless path. $\Box$

\section{$H=$\{claw, $F$\} }
\label{sec:clawF}

	We are now ready to show which ({\em claw}, $F$)-free graphs $G$,  for $F$ a single four-vertex graph,  can guarantee that $X_G$ is $e$-positive; however, we begin with a negative result that will eliminate a number of cases:

\begin{Lemma}
A graph that is $H$-free for $H$ equal to  \{claw, diamond\}, \{claw, $K_4$\}, \{claw, $4K_1$\}, \{claw, $C_4$\}, \{claw, $2K_2$\}, or \{claw, co-claw\}, is not necessarily e-positive.
\label{thm:somethm}
\end{Lemma}


	\begin{proof}
		Consider the graph $G$ which is the $3$-sun graph  (Figure \ref{fig:ThreeSun}).
		Note that $G$ is $(claw, K_4)$-free, $(claw, diamond)$-free, $(claw, C_4)$-free, $(claw, 4K_1)$-free, $(claw, 2K_2)$-free, and {\em (claw, co-claw)}-free; however, recall from Stanley \cite{Sold} that the chromatic symmetric function for the $3$-sun is
		\[ X_G = 6e_{3,2,1} - 6e_{3,3} + 6e_{4,1,1} + 12e_{4,2} + 18e_{5,1} + 12e_{6}, \]
		implying $X_G$ is not e-positive.
	\end{proof}

		\begin{figure}
\begin{center}
\begin{tikzpicture} [scale = 1.25]
\tikzstyle{every node}=[font=\small]

\newcommand{\size}{1}

\newcommand{\threesun}{1}{
    \path (0, 0) coordinate (g1);
    \path (g1) +(-1.5 * \size, 0) node (g1_1){};
    \path (g1) +(-0.5 * \size, 0) node (g1_2){};
    \path (g1) +(0.5, 0) node (g1_3){};
    \path (g1) +(1.5 * \size, 0) node (g1_4){};
    \path (g1) +(0, 0.866 * \size) node (g1_5){};
    \path (g1) +(0, 1.866* \size) node (g1_6){};
    \foreach \Point in {(g1_1), (g1_2), (g1_3), (g1_4), (g1_5), (g1_6)}{
        \node at \Point{\textbullet};
    }
    \draw   (g1_1) -- (g1_2)
            (g1_2) -- (g1_3)
            (g1_3) -- (g1_4)
            (g1_2) -- (g1_5)
            (g1_3) -- (g1_5)
            (g1_5) -- (g1_6);
    \path (g1) ++(0,-\size / 2) node[draw=none,fill=none] { {\large $3-sun$}};
}

\end{tikzpicture}
\end{center}
\caption{Example from Stanley \cite{Sold} of a graph that is not $e$-positive.} \label{fig:ThreeSun}

\end{figure}
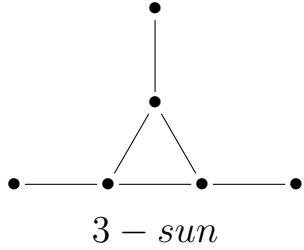

	\begin{Theorem}\label{thm:claw_paw_free}
		If $G$ is $(claw, paw)$-free, then $X_G$ is $e$-positive.
			\end{Theorem}
	\begin{proof}
		We prove by induction on the number of vertices.  Let $G$ be $(claw, paw)$-free. If $G$ is disconnected, then by the induction hypothesis, each component of $G$ is $e$-positive and so by Lemma~\ref{lem:chromatic_disjoint_union}, $G$ is $e$-positive. So we may assume $G$ is connected. By Theorem~\ref{thm:olariu} (Olariu's theorem),  $G$ is triangle-free or is a complete multipartite graph. By Lemma \ref{lem:claw_triangle_free}, $G$ is either $P_k$ or $C_k$, or $G$ is a complete multipartite graph.

		If $G$ is $P_k$ or $C_k$, then by Lemma~\ref{lem:e_path} or Lemma~\ref{lem:e_cycle}, $G$ is $e$-positive. If $G$ is a complete multipartite graph, since $G$ is also claw-free we have $\alpha(G) \le 2$.  It then follows from Lemma~\ref{lem:cotrianglefreeepos} that $G$ is $e$-positive.
		Therefore, $(claw, paw)$-free graphs are $e$-positive.
	\end{proof}


	\begin{Theorem}\label{thm:claw_copaw_free}
		If $G$ is  (claw, co-paw)-free, then $X_G$ is $e$-positive.
	\end{Theorem}
	\begin{proof}
			We prove by induction on the number of vertices. Let $G$ be ({\em claw, co-paw})-free.  
			If $\overline{G}$ is disconnected, then $G$ is the join of two vertex sets $V_A$ and $V_B$. Since $G$ is claw-free, we have $\alpha(V_A) \le 2$ and $\alpha(V_B) \le 2$, and thus $\alpha(G) \le 2$, and so $G$ is $e$-positive by Lemma~\ref{lem:cotrianglefreeepos}. Thus, we may assume $\overline{G}$ is connected. 
			
			 By Theorem~\ref{thm:olariu}, $G$ is the graph  $\overline{P_k}$ or $\overline{C_k}$, or $G$ is the complement of a complete bipartite graph.  If $G$ is $\overline{P_k}$ or $\overline{C_k}$, then $\alpha(G) \le 2$, which implies $G$ is $e$-positive by Lemma~\ref{lem:cotrianglefreeepos}. If $G$ is the complement of a complete bipartite graph, then $G$ is a disjoint union of cliques. So, $G$ is $e$-positive by Lemma~\ref{lem:chromatic_disjoint_union}  and Lemma~\ref{lem:chromatic_complete}.

	\end{proof}

The last four-vertex graph $F$ to consider is the co-diamond.  However, there is no straightforward decomposition that allows us to handle this case.  To this end,  we will make some observations about the structure of {\em (claw, co-diamond)}-free graphs in the next section.
\section{The structure of (claw, co-diamond)-free graphs}

In this section, we study the structure of (claw, co-diamond)-free graphs. 

\begin{Lemma}\label{lem:disconnected-co-diamond-free}
	Let $G$ be a co-diamond-free graph that is disconnected. Then 
	\begin{itemize}
		\item Each component of $G$ is a clique, or 
		\item $G$ is the union of a complete multipartite graph and the graph $K_1$.
	\end{itemize}
\end{Lemma}
\begin{proof}
	Let $G$ be a co-diamond-free graph that is disconnected. We may assume some component $C$ of $G$ is not a clique, for otherwise we are done. So $C$ contains a $P_3$ as induced subgraph. If $G$ has at least three components, then $G-C$ contains a $2K_1$ and this $2K_1$ together with an edge of $C$ induce a co-diamond, a contradiction. So $G$ has exactly two components. Let $C'$ be the component different from $C$. If $C'$ contains an edge, then this edge and the two end-points of the $P_3$ in $C$ from a co-diamond. So $C'$ is the graph $K_1$. The component $C$ cannot contain a co-$P_3$ as induced subgraph, for otherwise this co-$P_3$ and $C'$ form a co-diamond. Since a connected graph without co-$P_3$ must be a complete multipartite graph, the Lemma follows.  
\end{proof}
\begin{Lemma}\label{lem:disconnected-e-positive}
	If a (claw, co-diamond)-free graph $G$ is disconnected, then $G$ is $e$-positive.
\end{Lemma}
\begin{proof}
	Let $G$ be a (claw, co-diamond)-free graph that is disconnected. We will rely on Lemma~\ref{lem:disconnected-co-diamond-free}. If each component of $G$ is a clique, the $G$ is $e$-positive by Lemma~\ref{lem:chromatic_disjoint_union} and Lemma~\ref{lem:chromatic_complete}. So we may assume $G$ is the union of a complete multipartite graph $M$ and the graph $K_1$. Since $G$ is claw-free, we have $\alpha(M) \leq 2$, and so $M$ is $e$-positive by Lemma~\ref{lem:cotrianglefreeepos}. It follows from Lemma~\ref{lem:chromatic_disjoint_union} that $G$ is $e$-positive.
\end{proof}
\begin{Lemma}\label{lem:big-alpha}
	If a (claw, co-diamond)-free graph $G$ has $\alpha(G) \geq 4$, then $G$ is $e$-positive.
\end{Lemma}
\begin{proof}
	We use a result of Lozin and Malyshev (Lemma 9 in \cite{LM}) that shows a (claw, co-diamond)-free graph $G$ with $\alpha(G) \geq 4$ must be a stable set. For the sake of completeness, we are going to provide a proof of this statement. Let $G$ be a (claw, co-diamond)-free graph $G$ with $\alpha(G) \geq 4$. Consider a stable set $S$ that has at least four vertices and is maximal (that is, not belonging to a larger stable set). Consider any vertex $v \in G - S$, if such a vertex exists. Vertex $v$ must have a neighbor in $S$ by the maximality of $S$. If $v$ has at least three neighbors in $S$, then $G$ contains a claw, otherwise $G$ contains a co-diamond. So $G$ must be $S$. Clearly, a stable set is $e$-positive. 
\end{proof}	
	So we only need to consider (claw, co-diamond)-free graphs $G$ that are connected and have $\alpha(G) \leq 3$ ($4K_1$-free).
	If $G$ is triangle-free, then by Lemma~\ref{lem:claw_triangle_free}, $G$ is $e$-positive. Similarly, if $G$ is co-triangle-free, then $G$ is $e$-positive. So, we have $\alpha(G) = 3$. Suppose $G$ contains a triangle and co-triangle. Let us name the vertices of the co-triangle $a,b,c$. Since $G$ is co-diamond-free, any vertex not in the co-triangle must be joined to two vertices in the co-triangle (if it is joined to none we get a $4K_1$, if it is joined to one then we get a co-diamond, and if it is joined to three then we get a claw). Let the set $S_{x,y}$ be the vertices adjacent to the two vertices $x,y$ of the co-triangle. This means that ({\em claw, co-diamond})-free graphs which are not known to be $e$-positive from the previous theorems have the structure depicted in Figure \ref{fig:claw-codiamond-free} (we will call such graphs {\it peculiar}).
	

	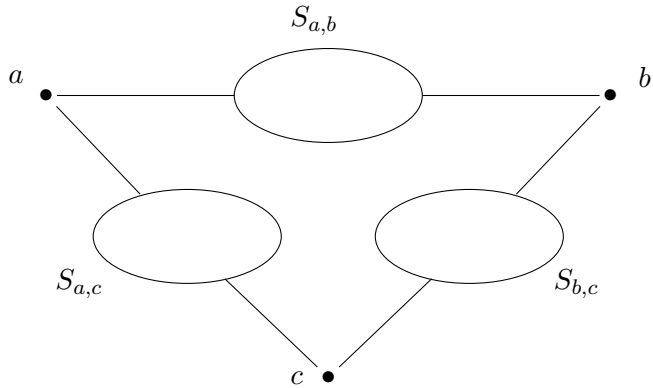
\begin{figure}[h]
		\begin{center}
\begin{tikzpicture} [scale = 1.25]
\tikzstyle{every node}=[font=\small]

\newcommand{\size}{1}

\newcommand{\ClawCodiamondStructure}{1}{
    \path (0, 0) coordinate (g1);
    \path (g1) +(-3 * \size, 0) node (g1_1){}; 
    \path (g1) +(0, -3 * \size) node (g1_2){}; 
    \path (g1) +(3 * \size, 0) node (g1_3){}; 
    \path (g1) +(-2 * \size, -1.05 * \size) node (g1_4){}; 
    \path (g1) +(2 * \size, -1.05 * \size) node (g1_5){}; 
    \path (g1) +(-1 * \size, 0) node (g1_6){}; 
    \path (g1) +(-1.1 * \size, -1.95 * \size) node (g1_7){}; 
    \path (g1) +(1.1 * \size, -1.95 * \size) node (g1_8){}; 
    \path (g1) +(1 * \size, 0) node (g1_9){}; 
    \foreach \Point in {(g1_1), (g1_2), (g1_3)}{
        \node at \Point{\textbullet};
    }
    \draw (-1.5 * \size, -1.5 * \size) ellipse (1cm and 0.5cm);
    \draw (1.5 * \size, -1.5 * \size) ellipse (1cm and 0.5cm);
    \draw (0, 0) ellipse (1cm and 0.5cm);
    \draw   (g1_1) -- (g1_4.center)
            (g1_1) -- (g1_6.center)
            (g1_2) -- (g1_7.center)
            (g1_2) -- (g1_8.center)
            (g1_3) -- (g1_5.center)
            (g1_3) -- (g1_9.center);

    \node[text width=1cm] at (-3,0.2) {$a$};
    \node[text width=1cm] at (3.7,0.2) {$b$};
    \node[text width=1cm] at (0,-3) {$c$};
    \node[text width=1cm] at (0,0.8) {$S_{a,b}$};
    \node[text width=1cm] at (-2.5,-2) {$S_{a,c}$};
    \node[text width=1cm] at (2.8,-2) {$S_{b,c}$};
}

\end{tikzpicture}
\end{center}
		\caption{The structure of a connected ({\em claw, co-diamond})-free graph that is not known to be $e$-positive.  The three  black vertices are the co-triangle. Each oval represents a subgraph, with each vertex in subgraph being joined to the two corresponding vertices of the co-triangle. At least two ovals are non-empty.  }
		\label{fig:claw-codiamond-free}
	\end{figure}


		Now that we understand  the structure of ({\em claw, co-diamond})-free graphs, we can consider adding additional restrictions to the graph.  In the following theorems we explore the e-positivity question for {\em (claw, co-diamond, $F$})-free graphs for $F$ a four-vertex graph.

		\begin{Theorem}
			If $G$ is (claw, co-diamond, $P_4$)-free, (claw, co-diamond, paw)-free, or (claw, co-diamond, co-paw)-free, then $X_G$ is $e$-positive.
		\end{Theorem}
		\begin{proof}
		This follows directly from Theorems \ref{thm:P4_free}, \ref{thm:claw_paw_free}, and \ref{thm:claw_copaw_free}.
		\end{proof}

		\begin{Theorem}\label{claw-diamond_co-diamond}
			If $G$ is  (claw, co-diamond, diamond)-free or (claw, co-diamond, co-claw)-free, then $X_G$ is $e$-positive.
		\end{Theorem}
		\begin{proof} 
			Let $G$ have a structure defined in Figure \ref{fig:claw-codiamond-free}. If $G$ is ({\em claw, co-diamond, diamond})-free, then the vertices in each oval is a stable set (if $S_{x,y}$ contains an edge, then this edge and $\{x,y\}$ form a diamond). Similarly, if $G$ is ({\em claw, co-diamond, co-claw})-free, then each oval is a stable set (if $S_{x,y}$ contains an edge, then this edge form a co-claw with $\{x,z\}$ where $z$ is the vertex of the co-triangle different from $x$ and $y$).  However, if there are three or more vertices in a oval then there  exists an induced claw. Therefore each oval can have a maximum of two vertices.
			
			Since $G$ is connected, there must be two non-empty ovals. There are seven base cases to consider when placing vertices in each oval. For each of these base graphs, the number of possible edges (between the ovals) each graph can have are also noted:
			\begin{itemize}
				\item one vertex in two ovals (one possible edge)
				\item two vertices in an oval, one vertex in another oval (two possible edges)
				\item two vertices in two ovals (four possible edges)
				\item one vertex in all three ovals (three possible edges)
				\item two vertices in an oval, one vertex in the remaining two ovals (five possible edges)
				\item two vertices in two ovals, one vertex in the remaining oval (eight possible edges)
				\item two vertices in all three ovals (twelve possible edges)
			\end{itemize}

			For each of the cases, there are a finite number of possible graphs. For each of these graphs, we check if the graph is ({\em claw, co-diamond, diamond})-free/({\em claw, co-diamond, co-claw})-free and if it has an induced triangle. If it meets these requirements, we then check whether the $X_G$ is $e$-positive.

			Since the number of total cases is relatively small, they are easily verifiable through the use of a computer program. Combining every possible case listed above, there are only 36 graphs (including isomorphic cases) which are ({\em claw, co-diamond, diamond})-free or ({\em claw, co-diamond, co-claw})-free and have an induced triangle, and all of them are $e$-positive. 
		\end{proof}

		Thus, for ({\em claw, co-diamond, $H$})-free graphs, there are only three remaining cases to consider: ({\em claw, co-diamond, $C_4$})-free, ({\em claw, co-diamond, $2K_2$})-free, and ({\em claw, co-diamond, $K_4$})-free. \\

\begin{Observation}\label{lem:C4}
	If a peculiar graph $G$ is ({\em claw, co-diamond, $C_4$})-free, then each oval induces  a clique.  
\end{Observation}
\begin{proof}
Suppose $G$ is ({\em claw, co-diamond, $C_4$})-free. If $S_{x,y}$ contains two non-adjacent vertices $u,v$, then the set $\{x,y,u,v\}$ induces a $C_4$.   
\end{proof}
Consider the graph $G$ of Figure~\ref{fig:claw-codiamond-free}. We will call $G$ the {\em generalized pyramid} if every oval is a clique and there are all edges between any two ovals, see Figure~\ref{fig:generalized-pyramid}. 
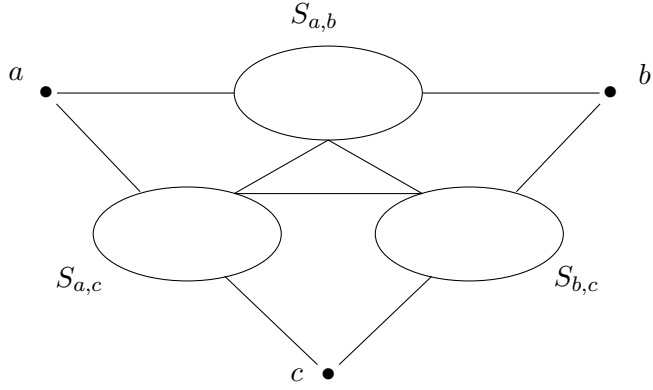
\begin{figure}
\begin{center}
\begin{tikzpicture} [scale = 1.25]
\tikzstyle{every node}=[font=\small]

\newcommand{\size}{1}

\newcommand{\GeneralizedPyramid}{1}{
    \path (0, 0) coordinate (g1);
    \path (g1) +(-3 * \size, 0) node (g1_1){}; 
    \path (g1) +(0, -3 * \size) node (g1_2){}; 
    \path (g1) +(3 * \size, 0) node (g1_3){}; 
    \path (g1) +(-2 * \size, -1.05 * \size) node (g1_4){}; 
    \path (g1) +(2 * \size, -1.05 * \size) node (g1_5){}; 
    \path (g1) +(-1 * \size, 0) node (g1_6){}; 
    \path (g1) +(-1.1 * \size, -1.95 * \size) node (g1_7){}; 
    \path (g1) +(1.1 * \size, -1.95 * \size) node (g1_8){}; 
    \path (g1) +(1 * \size, 0) node (g1_9){}; 
    \path (g1) +(0, -0.5 * \size) node (g1_10){}; 
    \path (g1) +(-1 * \size, -1.07 * \size) node (g1_11){}; 
    \path (g1) +(1 * \size, -1.07 * \size) node (g1_12){}; 
    \foreach \Point in {(g1_1), (g1_2), (g1_3)}{
        \node at \Point{\textbullet};
    }
    \draw (-1.5 * \size, -1.5 * \size) ellipse (1cm and 0.5cm);
    \draw (1.5 * \size, -1.5 * \size) ellipse (1cm and 0.5cm);
    \draw (0, 0) ellipse (1cm and 0.5cm);
    \draw   (g1_1) -- (g1_4.center)
            (g1_1) -- (g1_6.center)
            (g1_2) -- (g1_7.center)
            (g1_2) -- (g1_8.center)
            (g1_3) -- (g1_5.center)
            (g1_3) -- (g1_9.center)
            (g1_10.center) -- (g1_11.center)
            (g1_10.center) -- (g1_12.center)
            (g1_11.center) -- (g1_12.center);

    \node[text width=1cm] at (-3,0.2) {$a$};
    \node[text width=1cm] at (3.7,0.2) {$b$};
    \node[text width=1cm] at (0,-3) {$c$};
    \node[text width=1cm] at (0,0.8) {$S_{a,b}$};
    \node[text width=1cm] at (-2.5,-2) {$S_{a,c}$};
    \node[text width=1cm] at (2.8,-2) {$S_{b,c}$};
}

\end{tikzpicture}
\end{center}
\caption{The generalized pyramid graph.} 
\label{fig:generalized-pyramid}
\end{figure}

\begin{Lemma}\label{lem:2K2}
	If a peculiar graph $G$ is ({\em claw, co-diamond, $2K_2$})-free, then $G$ is the generalized pyramid. 
\end{Lemma}
\begin{proof}
	Let $G$ be a peculiar ({\em claw, co-diamond, $2K_2$})-free graph. Suppose there are non-adjacent vertices $u,v$ with $u \in S_{a,b}$ and $v \in S_{b,c}$, then the set $\{a,u,c,v\}$ induces a $2K_2$. Thus, there are all edges between any two ovals. Suppose $S_{a,b}$ contains two non-adjacent vertices $u,v$. Assume without loss of generality that $S_{b,c}$ is non-empty. Then, for any vertex $x \in S_{b,c}$, the set $\{x,a,b,c\}$ induces a claw. So $G$ must be the generalized pyramid. 
\end{proof}
\begin{Lemma}\label{lem:K4}
	If a peculiar graph $G$ is ({\em claw, co-diamond, $K_4$})-free, then $G$ has at most 18 vertices.
\end{Lemma}
\begin{proof}
	Consider a peculiar graph $G$ is ({\em claw, co-diamond, $K_4$})-free. Consider  the oval $S_{a,b}$. Now, let $T$ be a set with three vertices in $S_{a,b}$. If $T$ is a stable set, the $T$ and $a$ induce a claw. If $T$ is a clique, then $T$ and $a$ induce a $K_4$. So $S_{a,b}$ contains no triangle, and no co-triangle. It follows from a folklore (Ramsey's theorem) that $S_{a,b}$ has at most five vertices. So $G$ has at most $5+5+5+3= 18$ vertices.  
\end{proof}
By Lemma~\ref{lem:K4}, there are a finite number of graphs to verify. Brute force checking is theoretically possible, but it is infeasible unless further properties of ({\em claw, co-diamond, $K_4$})-free graphs are discovered.


\section{$H=$\{claw, $F$, co-$F$\}}

	We can also consider which graphs are $e$-positive which are ({\em claw, $H$, co-$H$})-free. Referring back to Figure \ref{fig:4vertex-graphs}, we see that:
	\begin{itemize} 
		\item $K_4$ is complementary to $4K_1$
		\item diamond is complementary to co-diamond
		\item $C_4$ is complementary to $2K_2$
		\item paw is complementary to co-paw
		\item claw is complementary to co-claw
		\item $P_4$ is self-complementary. \\
	\end{itemize}
The case of the (claw, $P_4$)-free graphs was settled by Tsujie \cite{Tsujie}. The remaining five cases are settled by the results of this paper.
	\begin{Theorem} \label{thm:clawdiamondcodiamond}
		If $G$ is (claw, diamond, co-diamond)-free, then $X_G$ is $e$-positive.
	\end{Theorem}
	\begin{proof}
		By Theorem \ref{claw-diamond_co-diamond}, $X_G$ is $e$-positive. \\
	\end{proof}

	\begin{Theorem} \label{thm:clawpawcopaw}
		If $G$ is (claw, paw, co-paw)-free, then $X_G$ is $e$-positive.
	\end{Theorem}
	\begin{proof}
		By Theorem \ref{thm:claw_paw_free} and Theorem \ref{thm:claw_copaw_free}, $X_G$ is $e$-positive. \\
	\end{proof}

	\begin{Theorem} \label{thm:negative}
		If $G$ is (claw, $K_4, 4K_1$)-free, (claw, $C_4, 2K_2$)-free, or (claw, co-claw)-free, then $X_G$ may not be $e$-positive.
	\end{Theorem}
	\begin{proof}
	Use the $3$-sun graph in Figure \ref{fig:ThreeSun} and the same reasoning as in Lemma \ref{thm:somethm}.
	\end{proof}

\section{Concluding Remarks}

We have considered the e-positivity question for the chromatic symmetric functions for two different classes of $H$-free graphs: $H = \{claw, F\}$ and $H=$\{{\em claw}, $F$, {\em co-}$F$\},  where  $F$ is a four-vertex graph.  In the case of $H = \{claw, F\}$ there are ten different cases to consider, of which one of these, the case  $H=\{ claw, P_4\}$, was previously proved $e$-positive by Tsujie \cite{Tsujie}.   Here we settle eight of the remaining cases, leaving  just the case of $H=$\{{\em claw, co-diamond}\}.
Based on our results and our investigations towards a result for  $H=$\{{\em claw, co-diamond}\}, we propose the following two open problems:

\begin{OpenProblem}
Are generalized pyramids $e$-positive?
\end{OpenProblem}

\begin{OpenProblem}
Are (claw, co-diamond)-free graphs $e$-positive?
\end{OpenProblem}

\section*{Acknowledgements}

This work was supported by the Canadian Tri-Council Research Support
Fund.  The 1st and 2nd
authors (A.M.~Hamel and C.T.~Ho\`ang) were each supported by individual
NSERC Discovery Grants.  The 3rd
author (J.E.~Tuero) was supported by an NSERC
Undergraduate Student Research Award (USRA).

\end{document}